\documentclass[10pt,draft,reqno]{amsart}
     \makeatletter
     \def\section{\@startsection{section}{1}%
     \z@{.7\linespacing\@plus\linespacing}{.5\linespacing}%
     {\bfseries%\normalfont\scshape
     \centering
     }}
     \def\@secnumfont{\bfseries}
     \makeatother
\usepackage{amsthm}
\usepackage{amsfonts}
\usepackage{color}
\usepackage{graphicx}
\usepackage{amssymb}
\usepackage{amsmath,array}
\usepackage{pgf}
\usepackage{ragged2e}
\usepackage{multimedia}
\usepackage{xmpmulti}
\usepackage{enumerate}

\newtheorem{theorem}{Theorem}[section]
\newtheorem{lemma}[theorem]{Lemma}
\newtheorem{proposition}[theorem]{Proposition}

\newtheorem{remark}[theorem]{Remark}

\newtheorem{definition}[theorem]{Definition}
\numberwithin{equation}{section}
\newcommand{\Z}{\mbox{$\mathbb Z$}}

\newcommand{\R}{\mbox{$\mathbb R$}}     % For Real numbers

\newcommand{\C}{\mbox{$\mathbb C$}}
\newcommand{\T}{\mathcal{T}}
\newcommand{\Ss}{\mathcal{S}}
\newcommand{\B}{\mathcal{B(H)}}

\newcommand{\A}{\mathcal{A}}
\newcommand{\LL}{\mathcal{L}}
\newcommand{\HH}{\mathcal{H}}
\newcommand{\CC}{\mathcal{C}_r}
\begin{document}

\title[quantum  dynamical semi-groups with  QS dilation]{Some quantum  dynamical semi-groups with quantum  stochastic  dilation}

\author{Lingaraj Sahu}
\address{Lingaraj Sahu : Department of Mathematical Sciences, Indian Institute of Science Education and Research Mohali,
Knowledge City, Sector 81, SAS Nagar -140 306, India.}
\email{lingaraj@iisermohali.ac.in}
%\urladdr{http://www.math.univ.edu/$\sim$johndoe \bf(optional)}

\author{Preetinder Singh}
%\thanks{* This research is supported by the Doe Foundation}
\address{Preetinder Singh : Department of Mathematical Sciences, Indian Institute of Science Education and Research Mohali,
Knowledge City, Sector 81, SAS Nagar -140 306, India.}
\email{preetinder@iisermohali.ac.in}


\begin{abstract}
 We consider the GNS Hilbert  space $\HH$ of a uniformly hyper-finite  $C^*$- algebra  and study
a class of   unbounded Lindbladian arises from commutators. Exploring the local structure of UHF algebra,
we have shown  that the associated Hudson-Parthasarathy type quantum stochastic differential equation admits a unitary solution.  The vacuum expectation  of homomorphic  co-cycle,
implemented by the Hudson-Parthasarathy flow, is conservative and gives the minimal semi-group associated with  the formal Lindbladian. We also associate conservative  minimal semi-groups  
to another class   of  Lindbladian   by   solving  the corresponding   Evan-Hudson equation.
\end{abstract}
\maketitle

\section{Introduction}
Quantum dynamical semi-groups (QDS) appear  naturally  when one studies the evolution of irreversible open quantum systems. QDS are non-commutative analogue
to Markov semi-groups in classical probability. For a  uniformly continuous semi-groups,  the generator is a bounded, conditionally completely positive (CCP) map. In \cite{lindblad}, Lindblad proved that  for  hyper-finite von Neumann algebras, which includes the case of $\B$, the generator $\LL$
of uniformly continuous QDS  can be written as $\LL(X) = \phi(X)+ G^*X+XG$,
where $\phi$ is a completely positive map and $G \in \B.$  In \cite{christensen} Christensen and Evans proved that for general $C^*$-algebras,
the generator of a uniformly continuous QDS exhibits the similar structure.

For the case of a strongly continuous  QDS, structure of the generator is not  well understood, Kato \cite{kato}  and
 Davies \cite{davies}    studied some  unbounded operators   or forms  similar to above on $\B$   and   gave
 a construction of  one-parameter semi-groups,  so-called minimal semi-group.  Under certain assumptions,  Davies in \cite{davies1}  showed that the unbounded generator have a similar  form as for the bounded case, thus extends the Lindblad's result to strongly continuous QDS.
However, these semi-groups need  not preserve the identity, i.e.,  need not be Markov.  Generally  such  unbounded operator   or form referred   as Lindbladian.
Starting with a Lindbladian,  a similar construction of a minimal semi-group was done for any von Neumann algebra in \cite{kbs}.

In this article, we have considered  Hudson-Parthasarathy (HP)     quantum stochastic differential equation   associated   with a model of unbounded Lindbladian  and construct the QDS
by
taking  vacuum expectation.
% For  quantum stochastic  calculus see \cite{krp1}.
There are  various attempts to  study  HP quantum stochastic differential equation     with unbounded coefficients, for example see \cite{ff, kbs}
and references therein.
%In  \cite{mohari}, Mohari gave an approach to
% construct the solution  and  under certain assumption  proved   isometry and co-isometry  properties of the HP flows.\\

In section $2$, we discuss briefly QDS  and some results  of  quantum stochastic  calculus and quantum stochastic differential equations (QSDE) with bounded operator coefficients.
In section $3,$   a class of  unbounded  Lindblad  form are   defined on the GNS space of UHF $C^*$-algebra and properties of structure maps are studied.
Finally, exploring the local structure of UHF algebra, it is shown that the associated HP equation
admits a unitary solution.
This implies that the expectation semi-group of  the homomorphic  co-cycle implemented by this unitary   is
conservative and therefore the unique (also minimal) $C_0$-contraction semi-group associated with the given form.
The model is very special and hence simple enough to allow the construction of the minimal semi-group,
without any of the machineries of the  abstract theories, mentioned earlier in the introduction.
\section{Preliminaries}
%\subsection{Review on Quantum Dynamical Semi-groups}
Let $\HH$ be a separable Hilbert space and  $\B$ denotes the von Neumann algebra of bounded linear operators on $\HH$.
\begin{definition}
A \textit{quantum dynamical semi-group} on a von Neumann algebra $\A \subseteq \B$ is a semi-group  $\T = (\T_t)_{t\geq0}$  of completely positive maps
on $\A$ with the following properties:
\begin{enumerate}[$(i)$]
\item $ \T_t(I) \leq I$, for all $t \geq 0 $.
%\item \text it{Completely positive} For all positive integers $n$ and all %finite sequences ${(xi)^n}_{i=1}, {(_r)^n}_{j=1}$ of elements of $\A$ we %have \\
 \item $ \T_t$ is a ultra-weakly continuous operator i.e. normal for all $t \geq 0$.
\item for each $a \in \A$, the map $t \rightarrow \T_t(a)$ is continuous  with respect to the ultra-weak topology  on $\A$.
\end{enumerate}
\label{Def1}
\end{definition}
A QDS is called \textit{Markov} or \textit{Conservative} if $\T_t(I)=I$ for every $t.$
% Further, if  $t\mapsto   T_t$ is continuous with respect to the norm topology, the  QDS is called \textit{uniformly continuous}.
\begin{theorem} \cite{lindblad, christensen, krp1, kbs}
A bounded map $\LL$ on the  von Neumann algebra $\B$ is the infinitesimal generator of a uniformly
continuous QDS $(\T_t)_{t\geq0}$ if and only if it can be written as $$\LL(X)=\sum\limits_{n=1}^\infty L_n^*XL_n+G^*X+XG, ~
\mbox{for all}~ X \in \B,$$  where $L_n$'s and $G$ are in $\B$ and the  series on the right side converges strongly, with $G$ generator of a contraction semi-group in $\HH$.  The QDS is Markov if and only if $$Re(G)= -\dfrac{1}{2}\sum\limits_{n=1}^\infty L_n^*L_n. $$
\end{theorem}
%\text bf{Assumption 2.3.}
For more general QDS, the generator can be understood as one coming from a similarly defined quadratic form  on $\HH,$ e.g., for $X\in \mathcal B(\HH),$
\begin{equation}
 \langle u,\LL(X)v\rangle  \equiv \langle u,XGv \rangle+\langle Gu,Xv \rangle + \sum\limits_{n=1}^\infty\langle L_nu,XL_nv \rangle
\label{QF}
\end{equation}
where these  $ L_n$  and $G$ are unbounded operators, $G$  is the generator of a $C_0$-contraction semi-group in $\HH$  such that
$Dom(G) \subseteq Dom(L_n)$, for each $n$  and
\begin{equation}
 \langle u,\LL(I)v\rangle  \equiv \langle u,Gv \rangle+\langle Gu,v \rangle + \sum\limits_{n=1}^\infty\langle L_nu,L_nv \rangle = 0,
\label{Unitalsgp}
\end{equation}
for all $ u,v \in Dom(G).$

Conversely, let $G$ be the generator (not necessarily bounded) of a $C_0$-contraction semi-group in $\HH$ and $L_n$
be a family of closed densely defined linear operators in $\HH$ with $Dom(G) \subseteq Dom(L_n)$  and let $\LL$ define formally by \eqref{QF}
 satisfies  \eqref{Unitalsgp}.
%Starting with the form like \ref{Q}, we can  associate a QDS in a sense that \refer{Seq} holds
Then the aim  is to construct a canonical (minimal)  semigroup  associated  with the formal Lindbladian $\LL,$  for some results in this direction  see \cite{davies,kato,kbs}.
%\noindent We  state the Lumer-Phillips Theorem  which will be used in our construction in coming section.
%\begin{theorem} \cite{kato}
%For a densely defined, dissipative operator $A$ on a Banach space $X$ the following statements are equivalent.
%\begin{enumerate}[$(i)$]
%\item The closure $\bar{A}$ of $A$  generates a $C_0$ contraction semi-group.
%\item Range of the operator $(\lambda -A)$ is dense in $X$ for some $\lambda > 0$.
%\end{enumerate}
%\label{Thm2}
%\end{theorem}
%\subsection{Brief Review on Quantum Stochastic Calculus}

We conclude this section with a brief discussion of  Quantum stochastic calculus developed by Hudson and Parthasarathy. We state
a result of existence and uniqueness of unitary solution for QSDE. For detail see \cite{krp1, kbs}.

For  a separable Hilbert space $\HH$, let $\Gamma_{sym}(\HH)$ denotes the symmetric Fock space over $\HH$. For any $u \in \HH $, we denote by $e(u)$, the exponential vector in $\Gamma_{sym}(\HH)$ associated with $u$:
$$e(u)= \bigoplus\limits_{n \geq 0} \dfrac{1}{\sqrt{n}!}{u^{\otimes}}^n.$$
Given a contraction $T$ on $\HH$,  the second quantization $\Gamma(T)$ on $\Gamma_{sym}(\HH)$ is defined  by $\Gamma(T)e(u)=e(Tu)$ and extends to a contraction
on $\Gamma_{sym}(\HH).$  Moreover, if $T$ is an isometry (respectively unitary), then so is $\Gamma(T)$.\\

Let us write $\Gamma_{sym}$ for the symmetric Fock space $\Gamma_{sym}(L^2(\R_+, \mathbf{k}))$, where $\mathbf{k}$ is  a  Hilbert space with an orthonormal basis $\{e_l:1\le l\le m\}.$
\noindent Following result  provides a nice criterion  for the existence of a unitary solution for an HP type  QSDE with bounded coefficients.
\begin{theorem} \cite{krp1}
Let  $H, \{L_i~;~ 1\leq i \leq m\},  \{S_j^i~;~ 1\leq i,j \leq m\}$ are bounded operators in $\HH$ satisfying the following conditions:
\begin{enumerate}[$(i)$]
\item
$H$ is self-adjoint.
%\item
%There exists a constant $c >0$ such that $\sum\limits_{i\geq 1} \Vert L_iu\Vert^2 \leq c^2\Vert u\Vert ^2$ for all $u \in \HH$.
\item
$\sum\limits_{1\leq i,j \leq m}S_j^i \otimes \vert e_i\rangle\langle e_j\vert$ is a unitary operator in $\HH \otimes \mathbf{k}$.
\end{enumerate}
\noindent Define
\begin{equation}
L_j^i=
\begin{cases}
S_j^i-\delta_j^i & \text{if }~ 1\leq i,j  \leq m;\\
L_i & \text{if }~ 1\leq i \leq m, j=0;\\
-\sum\limits_{1\leq k \leq m} L_k^*S_j^k & \text{if }~ 1\leq j \leq m , i=0;\\
-(\iota H + \dfrac{1}{2}\sum\limits_{1\leq k \leq m} L_k^*L_k ) & \text{if}~ i=j=0;\\
\end{cases}
\end{equation}

 \noindent where $\delta_j^i$ is the Kronecker's delta function.
Then there exists a unique unitary process $U_t$ satisfying  the QSDE on $\HH \otimes \Gamma_{sym}$

\begin{equation}
U_t = I+ \sum\limits_{i,j=0}^m\int\limits_0^tU_s  L_j^i \Lambda_i^j(ds),
\label{QSDE}
 \end{equation}
 where $\Lambda_0^0$ is time, for $i,j\ge 1, \Lambda_i^j$ is conservation, $\Lambda_i^0$ is creation and $\Lambda_0^j$ is annihilation
 processes.

\label{ThmQSDE}
\end{theorem}

\noindent Let  $(U_t)_{t\geq 0}$ be a  unitary process satisfying  \eqref{QSDE}. Then the  family of homomorphisms $\{J_t:t\ge 0\}$  defined by
$$J_t(X)= U_t^*(X\otimes I)U_t,~ X \in \B. $$ satisfies the QSDE
\begin{equation}
J_t (X)= X\otimes I+ \sum\limits_{i,j=0}^m\int\limits_0^tJ_s  \theta_j^i (X) \Lambda_i^j(ds)
\end{equation}
where $$ \theta _j^i (X)= XL_j^i+ {(L_i^j)}^* X+ \sum\limits_{k=1}^m  {(L_i^k)}^* X L_j^k, ~\forall ~i,j \geq 0.$$
In particular $\theta_0^0$ is given by,
\begin{equation}
 \label{theta00}
\theta _0^0 (X)= \sum\limits_{k=0}^m  L_k^* X L_k + XL_0^0+ {L_0^0}^* X,
\end{equation}
is the
generator of a QDS $(\mathcal{T}_t)_{t\geq 0}$ and the homomorphic co-cycle $J_t$ dilates $\mathcal{T}_t$ in the sense that
\begin{equation}
\langle u e(0), U_t^* (X\otimes I)U_t  v e(0)\rangle = \langle u, \T_t(X) v \rangle , \forall  u, v \in \HH   ~{\mbox{and}}~ X \in \B. \end{equation}
The QDS $\mathcal{T}_t$ is called the vacuum expectation of $J_t$. This homomorphic co-cycle  $J_t,$   implemented by the  HP  flow $U_t,$   is known as an  HP  dilation of  the QDS $\mathcal T_t.$\\
Consider the time reversal operator $R_t$ on $L^2(\R_+, \mathbf{k})$ defined by
\begin{equation}
R_t(f)(s) := \begin{cases}
f(t-s) & \text{if }~ s\leq t;\\
f(s) & \text{if }~ s>t.\\
\end{cases}
\end{equation}
Observe that  $R_t$ is a self-adjoint unitary. Thus the  second quantization $\Gamma({R_t})$ is so.  For a bounded process $U_t$, define the dual process $\tilde{U_t}$ by $$ \tilde{U_t}:= (1\otimes \Gamma({R_t})){U_t}^* (1\otimes \Gamma({R_t})). $$
%where $\tilde{R_t}(u\otimes e(f) )= u\otimes \Gamma_{sym}(R_t)$.
\begin{proposition} \cite{kbs}
\noindent Let  $U_t$ be a bounded process satisfying  the QSDE \eqref{QSDE}. Then the dual process $\tilde{U_t}$ will satisfy the QSDE of the similar form given by,
$$ \tilde{U_t} = I+ \sum\limits_{i,j=0}^m\int\limits_0^t \tilde{U_s}  {L_i^j}^* \Lambda_i^j(ds). $$
\label{dual}
\end{proposition}
\section{Examples of  Quantum Dynamical Semi-groups}
In this section, we have  constructed a class  of
formal Lindbladian  on the GNS space of a  UHF $C^*$-algebra $\A.$  In Theorem \ref{main}  we will  show that  the associated  HP equation   admits a unitary  solution.

\noindent Let us consider the UHF $C^*$-algebra $\A$ as the $C^*$-inductive limit of the infinite tensor product of the matrix algebra  $M_N(\mathbb{C})$,
$$\A = \overline{\bigotimes\limits_{j \in \Z^d} M_N(\C)}^{c^*}.$$ The algebra $\A$ can be interpreted as inductive limit of full matrix algebras.
 For $x \in M_N(\C)$ and
$j \in \Z^d$, $x^{(j)}$ denotes an element of $\A$ with $x$ in the $j^{th}$ component and  identity everywhere else.
We shall call the elements of the form $\prod_{i\geq 1} x_i^{(j_i)}$ to be simple tensor elements in $\A$.
For a simple tensor element $x$ in $\A$, let $x_{(j)}$ be the $j^{th}$ component of $x$. Support `$supp(x)$'  of $x$ is defined to be the subset $\{j \in \Z^d; x_{(j)} \neq I \} $.
For a general element $x \in \A$ such that $x=\sum_{n=1}^\infty c_nx_n $ with  simple tensor elements $x_n$ and complex coefficients $c_n$, define $supp(x)=
\bigcup_{n\geq1} supp(x_n).$
For any $\Delta \subset \Z^d$, let $\A_\Delta$ denotes the $\ast$-sub algebra
generated by the elements of $\A$ with support in  $\Delta$. For $j = (j_1, j_2,\cdots , j_d ) \in \Z^d$, define $\Vert j\Vert= max\{\vert j_i\vert ~;~ 1\leq i \leq d  \}$ and set $\Delta_n = \{j\in \Z^d ; \Vert j \Vert \leq n\}$,  $\partial\Delta_n = \{j\in \Z^d ; \Vert j \Vert =n\}$.
We say an element $x\in \A$ is  local if $x\in \A_{\Delta_p}$  for some  $p \geq 1.$
The unique normalized trace $tr$ on $\A$ is given by $tr(x)=\dfrac{1}{N^n}Tr(x)$, for $x \in M_{N^n}(\C)$,
where $Tr$ denotes the matrix trace. The trace $tr$ is a faithful normal state on $\A$. The algebra $\A$ can be represented as vectors in the Hilbert space $\HH=L^2(\A, tr)$, the GNS Hilbert space for $(\A,tr)$, and as an element of  $\B$ by left multiplication. We write $\A_{loc}$  for   the dense $*$-algebra generated by local elements.

%Let denote  the  dense subspace  $\mathcal A_{loc}$   of  the   GNS  Hilbert space $\HH$  by $\A_{loc}.$
Consider a formal element of the type  $$r:=\sum\limits_{n=1}^\infty W_n { \mbox{ such that }}   \sum\limits_{n=1}^\infty \|W_n \|=\infty,  $$
where each  $W_n$   belongs to $\A_{\partial\Delta_n}.$
Let us denote formally $$ \sum\limits_{n=1}^\infty W_n^*  \mbox{ by }  r^*.$$

\noindent Now, if we set $\CC(x)=[r,x]=\sum\limits_{n=1}^\infty   [W_n,x]$ for $x\in \A_{loc},$ clearly it is well defined since  $[W_n,x]=0$  for all $n> m$ when $x$
is in finite dimensional algebra
$ \mathcal A_{\Delta_m} \subseteq \A_{loc}.$ Thus we have a densely defined linear operator $(\mathcal{C}_r, \A_{loc})$ in $\HH.$
%We will study the  Lindbladian form, for $X\in \B,$
%$$ \A_{loc} \times \A_{loc} \ni  (u,v)\mapsto  \langle \CC u,X \CC v \rangle -\frac{1}{2} \langle u,X   \CC^* \CC v \rangle -\frac{1}{2} \langle \CC^* \CC u,Xv \rangle $$  and aim is to associate  a  QDS.
%Let us observe  some useful  facts  on  $\CC.$
\begin{lemma}
Let $r$ be as above and  $n\ge 1.$ Consider the element $ r_n=\sum\limits_{k=1}^n W_k$ in $\A$ and define a bounded operator $\CC^{(n)}$ on
$\HH$ by setting $\CC^{(n)}(x)=[r_n,x]=\sum\limits_{k=1}^n   [W_k,x]$ for $x \in \A_{loc}.$
Then for each $n \geq 1$, $\A_{\Delta_n}$ is an invariant subspace for  $\CC$ and  $\CC^{(n)}$.
Also for $m\ge p$,
\begin{equation}
\CC|_{\A_{\Delta_p}}= \CC^{(m)}|_{\A_{\Delta_p}}=\CC^{(p)}|_{\A_{\Delta_p}}.
\label{crnm}
\end{equation}

\label{Lemma}
\end{lemma}
\begin{proof}
For $x$ is in $\A_{\Delta_n}, [W_k,x]=0$ for $k >n$. Thus $[r,x] = [r_n,x] \in \A_{\Delta_n} $ and $\A_{\Delta_n}$ is an invariant subspace under
$\CC$ and  $\CC^{(n)}$. Now for   $x\in  \mathcal A_{\Delta_p}$ and $m\ge p,$  it is easy to  see that
$ \CC(x)= \CC^{(m)}(x)=\CC^{(p)}(x).$
\end{proof}
\begin{proposition}
The operator $(\CC, \A_{loc})$ is closable.
\label{Prop1}
\end{proposition}
\begin{proof}
We shall show that $\A_{loc} \subseteq Dom(\mathcal{C}_{r}^*)$  and  for $x\in \A_{loc}, ~ \CC^*(x)=\mathcal{C}_{r^*}(x) = [r^*,x]$, thereby showing that the operator $\CC^*$ is densely defined and therefore  $(\CC, \A_{loc})$ is closable.  Indeed for  $x\in \A_{loc}$, there exists
$p \geq 1$ such that $ x \in \A_{\Delta_p}$. Define
$\Phi_x(y):= \left\langle x, \CC y\right\rangle \forall y \in \A_{loc}$. For   each $ y \in \A_{loc}$, there exists $m$ such that $ y \in \A_{\Delta_m}$.
As $ \{\A_{\Delta_n}\}$ is an increasing family of algebras, with no loss of generality, let us assume $m \geq p$.
% with domain  $\A_{loc}$ is bounded, then $x \in Dom(\CC^*)$. Let $x,y\in \A_{loc}$ then there exist $m,n$  such that  $ x\in \mathcal A_{\Delta_n}$  and  $ y\in \mathcal A_{\Delta_m}.$
Then by definition and property of trace and Lemma \ref{Lemma},
$$\Phi_x(y)= tr(x^*\CC y) = tr(x^*\CC^{(m)}y) =tr({(\mathcal{C}_{r^*}^{(m)}x)}^*y)
= \langle\mathcal{C}_{r^*}^{(p)}x, y \rangle = \langle\mathcal{C}_{r^*}x, y \rangle,$$ and thus $$|\Phi_x(y)| \leq \Vert \mathcal{C}_{r^*}x\Vert \Vert y \Vert , ~\forall ~ y\in \A_{loc}.$$
%and  thus $\Phi_x$ can be extended to a bounded linear functional on $\mathcal H$.
Thus  $x \in Dom(\CC^*)$ and
\begin{equation}
 \CC^*(x)= \mathcal{C}_{r^*}(x),  \forall x \in \A_{loc}.
 \label{cr*}
\end{equation}

%It is easy to see  $t\geq 0$, $tr([[r_t,x_1], x_2]) =   tr([x_1, [x_2, r_t]])$.\\ Now for $t \geq 0$ $$tr([[x, r_], y]) =   tr([[x,r_m, y]).$$
%\text bf{Case 1.} If $t \geq m$\\
%$tr([[x, r_t], y]) = $
\end{proof}

\noindent We denote by $\bar{\CC}$, the closure of a densely defined, closable operator $\CC$. Note here that for a operator $T$ on $\HH$, $T^*=\bar{T}^*,$ if $T$ is closable. Then by standard theorem of von Neumann, $\CC^*\bar{\CC}$ is a positive self-adjoint operator in $\HH$ and $Dom(\CC^*\bar{\CC})$ is a core for $\bar{\CC}$. Furthermore, the operator $G := -\dfrac{1}{2}\CC^*\bar{\CC}$ generates a $C_0$-contraction semi-group $\Ss_t$ in $\HH.$\\

\begin{proposition}
For  $n\ge 1,$ define   the bounded operator $G^{(n)}$ on $\HH$ by
$$ G^{(n)}:=
-\frac{1}{2} \mathcal{C}_{r^*}^{(n)}\CC^{(n)}.$$  Then each
$\A_{\Delta_n}$ is an invariant under  $G^{(n)}.$  Furthermore,
\begin{equation}
 G^{(m)}|_{\A_{\Delta_p}}=G^{(p)}|_{\A_{\Delta_p}} =  G|_{\A_{\Delta_p}}  \mbox{ if } m\geq p,
 \label{Restriction}
\end{equation}
\end{proposition}
\begin{proof}
By Lemma \ref{Lemma}, we have $\A_{\Delta_n}$ invariant  under
$\bar{\CC}$ and  $\CC^{(n)}$ and  for  $m\ge p,$ the identity
$ \CC|{\A_{\Delta_p}}= \CC^{(m)}|{\A_{\Delta_p}}=\CC^{(p)}|{\A_{\Delta_p}}$ holds. As
$ \CC^*(x)= \mathcal{C}_{r^*}(x),  ~\forall~ x \in \A_{loc} ,$   we have
$ \CC^*|{\A_{\Delta_p}}=    \mathcal{C}_{r^*}^{(m)}|_{\A_{\Delta_p}}=   \mathcal{C}_{r^*}^{(p)} |_{\A_{\Delta_p}}$
and  hence result follows.

\end{proof}
\begin{proposition}
The  subspace $\A_{loc}$ is a core for the  operator $G.$
\label{Prop4}
\end{proposition}
\begin{proof}
It is enough to  prove that the  subspace $\A_{loc}$  is invariant under the semi-group $\Ss_t.$ For  a vector
$x \in \A_{loc}$, there exists $n \geq 1$, such that $x \in {\A_{\Delta_n}}$.
Now by Lemma \ref{Lemma},  for any $k \geq 0$, $ G^k (x)= {G^{(n)}}^k (x) \in \mathcal A_{\Delta_n}$  and
it follows that the series  $\sum_{k\geq 0}  \frac{t^kG^k x}{k!}$  converges strongly in  $\mathcal A_{\Delta_n}.$
Therefore, we have,   $\Ss_tx= \Ss_t^{(n)}x \equiv  e^{tG^{(n)}}$  for $ x \in\A_{\Delta_n}.$ Thus, $\Ss_t$ leaves  $\A_{loc}$
invariant and  by Nelson's theorem \cite{nelson}, the core property follows.
\end{proof}

\noindent Now consider the sesquilinear form, Lindbladian,
$\LL(X)$ with the domain $ \A_{loc} \times \A_{loc} \subseteq Dom(G) \times Dom(G)$ given by
\begin{equation}
\langle u,\LL(X)v\rangle  \equiv \langle u,XGv \rangle+\langle Gu,Xv \rangle + \langle \bar{\CC} u,X \bar{\CC} v \rangle.
\label{Eq3}
\end{equation}
By definition of $G, $ it is clear that $\langle u,\LL(I)v\rangle =\langle u,Gv \rangle+\langle Gu,v \rangle + \langle \bar{\CC} u, \bar{\CC} v \rangle=0.$

Let $\A_{loc}\otimes\mathcal{E}$ be the linear span of $ \{x\otimes e(f)~:~ x \in \A_{loc},~ f \in L^2(\R_+, \C)\}$. Then the set $\A_{loc}\otimes \mathcal{E}$ is a dense subspace of  $\HH\otimes\Gamma_{sym} $.

\begin{theorem}
\label{main}
Consider the HP type QSDE in $\A_{loc}\otimes\mathcal{E}$
\begin{equation}
U_t = I+ \int\limits_0^t U_sGds + \int\limits_0^t U_s \bar{\CC} a^\dagger (ds) - \int\limits_0^t U_s \CC^* a(ds),
\label{Eq6}
\end{equation}
where $a^\dagger, a$ are creation and annihilation processes respectively. The QSDE \eqref{Eq6}
admits a unitary solution $U_t$. Moreover, the expectation semi-group $(\T_t)_{t\geq0}$ of the homomorphic co-cycle
$J_t(X)= U_t^*(X\otimes I)U_t$ is the unique (minimal) semi-group associated with the formal Lindbladian $\LL$ in \eqref{Eq3} and is conservative.

\end{theorem}
\begin{proof}
Recall that the UHF algebra $\A$ can be approximated by finite dimensional algebras, namely $\ A_{\Delta_n} = \prod\limits_{\Vert j\Vert \leq n}M_N(\C)$
and  $\A_{loc} = \bigcup\limits_{n=0}^\infty \A_{\Delta_n}.$
%Denote $\Gamma_{sym}(L^2(\R_+, \C))$ by $\Gamma_{sym}$. From the Lemma \ref{Lemma}, Equation \eqref{Restriction} and Theorem \ref{ThmQSDE}, the coefficients $G^{(n)}$, $\CC^{(n)}$ and ${\CC^{(n)}}^*$ are bounded, therefore for each $n$,
%we consider the subspace $\ A_{\Delta_n} \otimes\mathcal{E} \subseteq \A_{loc}\otimes\mathcal{E}_ \subseteq L^2(\A, tr)\otimes \mathcal{E}_$ ,
For $n\geq 0$, consider  the following QSDE in $\A_{loc} \otimes\mathcal{E},$
\begin{equation}
U_t^{(n)} = I+ \int\limits_0^t U_s^{(n)}G^{(n)}ds + \int\limits_0^t U_s^{(n)} \CC^{(n)} a^\dagger (ds) -\int\limits_0^t U_s^{(n)} {\CC^{(n)}}^* a(ds).
\label{Utn}
\end{equation}
By Theorem \ref{ThmQSDE}, the QSDE \ref{Utn} admits a unitary solution $U_t^{(n)}$ on $\HH \otimes \Gamma_{sym}$.\\
We now show that the operators $U_t^{(n)} $  satisfy some compatibility condition, that is for $n \geq m$ ,
\begin{equation}
U_t^{(n)}|_{\A_{\Delta_m}}= U_t^{(m)}|_{\A_{\Delta_m}}.
\label{compeq}
\end{equation}
Here the symbol $T|_{\A_{\Delta_m}}$ means the  restriction of $T$  to the subspace  $A_{\Delta_m}\otimes \Gamma_{sym}. $\\
Since these operators $\CC^{(m)}, {\CC^{(m)}}^*$ and $G^{(m)}$ leave $\A_{\Delta_m}$ invariant, the restriction   $U_t^{(m)}|_{\A_{\Delta_m}}$ satisfies  the following QSDE
in $ A_{\Delta_m}\otimes\mathcal{E}, $
\begin{equation}
U_t^{(m)}|_{\A_{\Delta_m}} = I|_{\A_{\Delta_m}} + \int\limits_0^t U_s^{(m)}|_{\A_{\Delta_m}}G^{(m)}|_{\A_{\Delta_m}} ds
\end{equation}
$$+ \int\limits_0^t U_s^{(m)}|_{\A_{\Delta_m}} \CC^{(m)}|_{\A_{\Delta_m}}  a^\dagger  (ds) -\int\limits_0^t U_s^{(m)}|_{\A_{\Delta_m}} {\CC^{(m)}}^*|_{\A_{\Delta_m}} a (ds).$$
\noindent
For $n\geq m$, consider the QSDE  in $\ A_{\Delta_m}\otimes\mathcal{E}$,
%For $ \otimes e(f) \in A_{\Delta_m}\otimes\mathcal{E} $
\begin{equation}
U_t^{(n)}|_{\A_{\Delta_m}} = I|_{\A_{\Delta_m}} + \int\limits_0^t U_s^{(n)}|_{\A_{\Delta_m}}G^{(n)}|_{\A_{\Delta_m}}ds
\label{Utn.on.m}
\end{equation}
 $$ + \int\limits_0^t U_s^{(n)}|_{\A_{\Delta_m}} \CC^{(n)}|_{\A_{\Delta_m}} a^\dagger  (ds) -\int\limits_0^t U_s^{(n)}|_{\A_{\Delta_m}} {\CC^{(n)}}^*|_{\A_{\Delta_m}}a (ds).$$
With reference to Lemma \ref{Lemma}, equation \eqref{Restriction} and Theorem \ref{ThmQSDE},
%the QSDE \eqref{Utn.on.m} becomes,
%\begin{equation}
%U_t^{(n)}|_{\A_{\Delta_m}} = I|_{\A_{\Delta_m}} + \int\limits_0^t U_s^{(n)}|_{\A_{\Delta_m}}G^{(m)}|_{\A_{\Delta_m}} ds +
%\int\limits_0^t U_s^{(n)}|_{\A_{\Delta_m}} \CC^{(m)}|_{\A_{\Delta_m}} a^\dagger  (ds)
%\end{equation}
%$$-\int\limits_0^t U_s^{(n)}|_{\A_{\Delta_m}} {\CC^{(m)}}^*|_{\A_{\Delta_m}} a(ds).$$
 the unitary processes $U_t^{(n)}|_{\A_{\Delta_m}}$ and $U_t^{(m)}|_{\A_{\Delta_m}}$
satisfy the same QSDE in  $ A_{\Delta_m}\otimes\mathcal{E} $. Therefore, by uniqueness of solution in Theorem \ref{ThmQSDE},  \eqref{compeq} follows.\\
Define $U_t$ on $\A_{loc}\otimes\mathcal{E}$ by setting
$$U_t(x\otimes e(f))= U_t^{(n)} (x\otimes e(f))~~~~ \text{if}~~~x \in \ A_{\Delta_n}$$
and extending linearly.    Since  the family  $U_t^{(n)}$   satisfies the  compatibility condition  \eqref{compeq},   $U_t$  is well defined on  $\A_{loc}\otimes\mathcal{E},$   and
for $x \in \ A_{\Delta_m}$  we have
\begin{equation}
U_t(x\otimes e(f))= U_t^{(m)} (x\otimes e(f))= U_t^{(n)} (x\otimes e(f)),\forall  n\geq m.
\label{Uteq}
\end{equation}
Hence $U_t^{(n)}$ converges strongly to $U_t$ on $\A_{loc} \otimes\mathcal{E}$ and
%that is for   $\xi \in \A_{loc}\otimes\mathcal{E}$,
%\begin{equation}
%U_t(\xi)=  \lim_{n\rightarrow\infty}  U_t^{(n)} (\xi),
%\label{stronglimiteq}
%\end{equation}
%with convergence  uniform in $t$. In fact, by \eqref{Uteq}, for a given $\xi \in \A_{loc} \otimes\mathcal{E}$ there exists $p$ such that
%\begin{equation}
% U_t(\xi) = U_t^{(n)}(\xi) = U_t^{(p)}(\xi), ~\forall ~ n \geq p, t\ge 0.
% \label{UtCG}
%\end{equation}
  $U_t$ extends to a contraction operator on $\HH \otimes\Gamma_{sym}$.
As  $\A_{loc} \otimes\mathcal{E}$  is dense in  $\HH \otimes\Gamma_{sym},$   \eqref{Uteq}  gives that  $U_t^{(n)}$ converges strongly to $U_t$ on $\HH \otimes\Gamma_{sym}$
as well and the limit $U_t$ is an isometry.\\
%To see this,
%let $\xi \in \HH \otimes\Gamma_{sym}$, then for given $\epsilon >0$ there exists $\xi' \in \A_{loc} \otimes\mathcal{E}$,
%and hence $x \in \A_{\Delta_p} \otimes\mathcal{E}$
%such that $ \Vert \xi-\xi'\Vert <\epsilon$.  Now consider,

%\begin{eqnarray*}
% \Vert  U_t^{(n)}(\xi)-U_t (\xi)\Vert =  \Vert  U_t^{(n)}(\xi)-U_t^{(n)}(\xi')+U_t^{(n)}(\xi')- U_t (\xi')+U_t(\xi')-U_t (\xi)\Vert \\
%\leq \Vert U_t^{(n)}(\xi)-U_t^{(n)}(\xi')\Vert +\Vert U_t^{(n)}(\xi')-U_t(\xi')\Vert+ \Vert  U_t (\xi')- U_t (\xi)\Vert \\
%\leq \|\xi-\xi'\|+ \Vert U_t^{(n)}(\xi')-U_t(\xi')\Vert +  \|\xi-\xi'\|.\\
%\end{eqnarray*}

%Since $\xi' \in \A_{loc} \otimes\mathcal{E}$, by \eqref{UtCG} there exist $p$ such that for any $n\geq p,$\\
%$\Vert U_t^{(n)}(\xi')-U_t(\xi')\Vert = 0$. Thus we have,\\
%$$\Vert  U_t^{(n)}(\xi)-U_t (\xi)\Vert <2  \epsilon, \hspace{1cm} \forall n > p.$$
%Thus the unitaries $U_t^{(n)}$ converge strongly and the limit $U_t$ is an isometry. \\
%\begin{equation}
%{U_t^{(n)}}^* = I+ \int\limits_0^t {G^{(n)}}^* {U_s^{(n)}}^*ds + \int\limits_0^t {\CC^{(n)}}^* {U_s^{(n)}}^*  a (ds) -\int\limits_0^t \CC^{(n)} {U_s^{(n)}}^*  a^\dagger(ds).
%\label{Utn*}
%\end{equation}
For  ${U_t^{(n)}}$, consider the dual process  ${\tilde{U_t}^{(n)}} = (1\otimes\Gamma(R_t)){U_t^{(n)}}^* (1\otimes\Gamma(R_t)) $.
Then by  Proposition \ref{dual},  $\{\tilde{U_t}^{(n)}\}$ satisfies the following  QSDE in $\A_{loc} \otimes\mathcal{E},$
\begin{equation}
{\tilde{U_t}^{(n)}} = I+ \int\limits_0^t {\tilde{U_s}^{(n)}} {G^{(n)}}^* ds + \int\limits_0^t {\tilde{U_s}^{(n)}} {\CC^{(n)}}^*  a (ds)
-\int\limits_0^t {\tilde{U_s}^{(n)}} \CC^{(n)} a^\dagger(ds).
\label{Utndual}
\end{equation}
The equation \eqref{Utndual} is identical to \eqref{Utn} except that ${\CC^{(n)}}$ is replaced by $-{\CC^{(n)}}$. So  similar  arguments yield  that the operators $\tilde{U_t}^{(n)}$ also satisfy the  compatibility condition
and   converge strongly to an isometry and because  $\tilde{U_t}^{(n)} $ and $\Gamma(R_t)$ are unitaries,
the sequence ${U_t^{(n)}}^*$ of  unitaries converges strongly and thus it must converge to $U_t^*$. Hence $U_t^*$ is an isometry, so $U_t$ is a unitary process. \\

It remains to prove that $U_t$ satisfies the QSDE \eqref{Eq6}. As $U_t$ is a unitary process, the quantum stochastic integral on the right-hand side of \eqref{Eq6} makes sense.
Thus, it is enough to establish that integrals on the right-hand side of \eqref{Utn} converge to integrals in \eqref{Eq6}.
For $xe(f)\in \A_{loc}\otimes \mathcal{E},$  we have
$$\Vert \int\limits_0^t (U_s^{(n)} G^{(n)}  - U_s G) ds (xe(f))\Vert \leq  \int\limits_0^t \Vert(U_s^{(n)} G^{(n)}  - U_s G) (xe(f))\Vert ds,$$
hence by \eqref{Restriction}  and
\eqref{Uteq}, it converges to $0.$
By estimates of   quantum stochastic integrals \cite{krp1}, we have \\
$\hspace*{1cm}\Vert \int\limits_0^t (U_s^{(n)} {\CC^{(n)}}   - U_s \CC)  a^\dagger (ds)(xe(f))\Vert^2$\\
\hspace*{2cm} $\leq  2 e^{\int\limits_0^t(1+|f(s)|^2)ds}\int\limits_0^t \Vert(U_s^{(n)} {\CC^{(n)}}
- U_s \CC ) xe(f)\Vert^2 (1+|f(s)|^2) ds.$\\
Therefore, by \eqref{crnm} and \eqref{Uteq},
 $$\lim_{n\rightarrow \infty}\Vert \int\limits_0^t (U_s^{(n)} {\CC^{(n)}}    - U_s \CC)  a^\dagger (ds)(xe(f))\Vert^2= 0.$$
Convergence of    annihilation   term  follows from a simpler estimate and   using  \eqref{cr*}, \eqref{crnm} and \eqref{Uteq}.
 Thus $U_t$  is a unitary solution to the QSDE \eqref{Eq6}.

Now let us consider the expectation semi-group $(\T_t)_{t\geq0}$ of the homomorphic co-cycle $J_t(\cdot)= U_t^*(\cdot\otimes I)U_t.$
As $U_t$ is a unitary  process,  the QDS $(\T_t)_{t\geq0}$ is conservative minimal semi-group associated with the form \eqref{Eq3}.
\end{proof}

\noindent We conclude by a small remark on the Lindbladian
\begin{equation}\label{WW}
\mathcal L(X)=  \frac{1}{2}\sum_{j=1}^\infty \{
 W_j^*\delta_j(X)  +  \delta_j^\dagger (X) W_j\},  \forall  X\in \mathcal A_{loc}
\end{equation}
where  $W_j\in \mathcal A_{\partial \Delta_j}, \delta_j(X)=[X,W_j],  \delta_j^\dagger (X)=  (\delta_j(X^*)) ^*=[W_j^*,X].$
Though  each component  $ W_j^*\delta_j(.)  +  \delta_j^\dagger (,) W_j $   are bounded  maps,    $\mathcal L$  is  unbounded due  to presence of infinitely many components   (like in \cite{mat}).
For $n\ge 1, $ define   a  bounded map $ \mathcal L^{(n)}(X)=  \frac{1}{2}\sum_{j=1}^n \{
 W_j^*\delta_j(X)  +  \delta_j^\dagger (X) W_j\},  \forall  X\in \mathcal A.$
 Note that for $X\in \A_{\Delta_n}, \delta_k(X)=\delta_j^\dagger (X) =0 $  and $\mathcal L^{(k)}(X)=  \mathcal L^{(n)}(X) $ for every $ k>n. $

 \begin{remark}
The HP equation  on $\mathcal H \otimes   \Gamma_{sym}   (L^2( \mathbb R_+, \mathbf k)),$  where $\mathbf k$ is a separable Hilbert space with   an orthonormal basis  $\{e_j:j \ge 1\},$
 $$U_t = I+ \int\limits_0^t U_s   (-\frac{1}{2}\sum_{j=1}^\infty  W_j^*W_j)ds + \sum_{j=1}^\infty  \int\limits_0^t U_s  W_j a_j^\dagger (ds) -  \sum_{j=1}^\infty  \int\limits_0^t U_s W_j^* a_j(ds)$$
may not make sense as $-\frac{1}{2}\sum_{j=1}^\infty  W_j^*W_j$ may have a trivial domain or not a generator of a $C_0$- semigroup on $\mathcal H$. However, there exist a  homomorphic co-cycle
$J_t:\A\rightarrow  \A^{''} \otimes \mathcal B(\Gamma_{sym})  $ satisfying the Evan-Hudson  equation, for $X\in \A_{loc},$
 $$J_t(X) = X\otimes I+ \int\limits_0^t J_s (\mathcal L (X)) ds + \sum_{j=1}^\infty  \int\limits_0^t J_s ( \delta_j(X)) a_j^\dagger (ds) +
 \sum_{j=1}^\infty  \int\limits_0^t J_s (\delta_j^\dagger (X))a_j(ds).$$
 The expectation semi-group $(\T_t)_{t\geq 0}$ of
  the homomorphic co-cycle $J_t$ is conservative minimal semi-group associated with the Lindbladian 
  \eqref{WW}.
  \end{remark}
  This can be seen  similarly as for HP  equation  in theorem \ref{main} by  constructing  $J_t$  as a strong limit of  homomorphic co-cycles
$\{J_t^{(n)}:\B\rightarrow  \B \otimes \mathcal B(\Gamma_{sym}) \} $ where  $J_t^{(n)}$ satisfies the  Evan-Hudson   equation, for $X\in \B,$
 $$J_t^{(n)}(X) = X\otimes I+ \int\limits_0^t J_s^{(n)} (\mathcal L^{(n)}(X)) ds + \sum_{j=1}^n \int\limits_0^t J_s^{(n)} ( \delta_j(X)) a_j^\dagger (ds) $$
 $$+
 \sum_{j=1}^n  \int\limits_0^t J_s^{(n)} (\delta_j^\dagger (X))a_j(ds)$$  with   bounded   structure maps  and  finite degree  of freedom   ( see \cite{krp1}).
In fact,  $J_t^{(n)}$
 takes $ \mathcal A_{ \Delta_n}$   to $\mathcal A_{ \Delta_n}\otimes \mathcal B(\Gamma_{sym} ),$ and for $X\in  \mathcal A_{ \Delta_n},$
$$
J_t(X)=J_t^{(m)}(X)=J_t^{(n)}(X), \forall  m\ge n.$$

\section*{Acknowledgment}
 Second author acknowledges partial  support  from Council of Scientific and Industrial Research, Govt. of India
 and National Board for Higher Mathematics, DAE,  Govt. of India.
 Both  authors thank Prof. K. B.  Sinha for his valuable suggestions.

\bibliographystyle{amsplain}

\end{document}